\documentclass[11pt]{article}
\usepackage{amsmath,amssymb,amsthm}
\usepackage{pstricks,pst-node}
\usepackage{graphicx}
\usepackage{multirow}
\usepackage[margin=1in]{geometry}

\usepackage{algorithmic}

\newtheorem{thm}{Theorem}[section]
\newtheorem{cor}[thm]{Corollary}
\newtheorem{lem}[thm]{Lemma}
\newtheorem{obs}[thm]{Observation}

\newtheorem{ques}{Question}

\theoremstyle{definition}
\newtheorem{dfn}[thm]{Definition}
\newtheorem{exa}[thm]{Example}
\newtheorem{alg}[thm]{Algorithm}

\newcommand{\F}{\mathcal{F}}
\newcommand{\G}{\mathcal{G}}
\newcommand{\C}{\mathcal{C}}
\newcommand{\D}{\mathcal{D}}
\renewcommand{\S}{\mathbb{S}}
\newcommand{\K}{\mathbb{K}}
\newcommand{\B}{\mathbb{B}}
\newcommand{\paw}{\mathrm{paw}}
\newcommand{\fourpan}{\textrm{$4$-pan}}
\newcommand{\cofourpan}{\textrm{co-$4$-pan}}

\title{Minimal forbidden sets for degree sequence characterizations}
\author{Michael D. Barrus\thanks{Research partially supported by Faculty Research Committee, Black Hills State University.}\\
Department of Mathematics\\
Brigham Young University\\
Provo, UT 84602 USA\\
\texttt{barrus@math.byu.edu}
\and
Stephen G. Hartke\thanks{Supported in part by National Science Foundation grant DMS-0914815.}\\
Department of Mathematics\\
University of Nebraska--Lincoln\\
Lincoln, NE 68588-0130 USA\\
\texttt{hartke@math.unl.edu}
}

\begin{document}
\maketitle

\begin{abstract}
Given a set $\F$ of graphs, a graph $G$ is $\F$-free if $G$ does not contain any member of $\F$ as an induced subgraph.  Barrus, Kumbhat, and Hartke~\cite{BarrusEtAl08} called $\F$ a \emph{degree-sequence-forcing (DSF) set} if, for each graph $G$ in the class $\C$ of $\F$-free graphs, every realization of the degree sequence of $G$ is also in $\C$. A DSF set is \emph{minimal} if no proper subset is also DSF.

In this paper, we present new properties of minimal DSF sets, including that every graph is in a minimal DSF set and that there are only finitely many DSF sets of cardinality $k$.  Using these properties and a computer search, we characterize the minimal DSF triples.
\medskip


\end{abstract}



\section{Introduction}

Given a collection $\F$ of graphs, a graph $G$ is \emph{$\F$-free} if $G$ contains no induced subgraph isomorphic to an element of $\F$, and the elements of $\F$ are \emph{forbidden subgraphs} for the class of $\F$-free graphs. 
Recently interest has developed in determining which small sets of induced subgraphs can be forbidden to produce graphs having special properties, such as $3$-colorability~\cite{Randerath2004}, traceability~\cite{GouldHarris99}, $k$-connected Hamiltonian~\cite{ChenEtAl12}, or the containment of a Hamiltonian cycle~\cite{FaudreeEtAl04}, 2-factor~\cite{AldredEtAl10}, or perfect matching~\cite{FujisawaEtAl2011,FujitaEtAl2006,OtaEtAl11,OtaSueiro2013}.

In this paper we investigate for which collections $\F$ of graphs the $\F$-free graphs are characterized by their degree sequences.  The \emph{degree sequence} $d(G)$ of a graph $G$ is the list $(d_1,d_2,\ldots,d_n)$ of its vertex degrees, written in non-increasing order, and any graph having degree sequence $d$ is a \emph{realization} of $d$. A graph class $\mathcal{C}$ has a \emph{degree sequence characterization} if it is possible to determine whether or not a graph $G$ is in $\mathcal{C}$ given only the degree sequence of $G$. Degree sequence characterizations of interesting graph classes are rare but valuable, as they often lead to very efficient recognition algorithms.

In~\cite{BarrusEtAl08} the authors and M.~Kumbhat studied the question of classifying graph classes with both forbidden subgraph and degree sequence characterizations. They defined a collection $\F$ of graphs to be \emph{degree-sequence-forcing} (or \emph{DSF}) if whenever any realization of a degree sequence $d$ is $\F$-free, every other realization of $d$ is $\F$-free as well.  In addition to several fundamental properties of DSF sets, \cite{BarrusEtAl08} also contains a characterization of the DSF sets with at most two graphs.

A DSF set $\F$ is \emph{minimal} if no proper subset of $\F$ is also DSF.  In this paper, we present new properties of minimal DSF sets.  We show that these sets are both plentiful and rare:  every graph belongs to a minimal DSF set (Theorem~\ref{thm: every graph in minimal}), but there are finitely many minimal DSF sets with $k$ graphs for any positive integer $k$ (Theorem~\ref{thm: bound on min DSF size}).  We also give a new characterization of all (including non-minimal) DSF sets (Theorem~\ref{thm: DSF iff D condition holds}).

The properties of minimal DSF sets give us an algorithmic framework to search for minimal DSF triples.  We perform a computer search for all such triples as described in Section~\ref{sec: the search}, and from this search obtain a complete list of all minimal DSF triples (Theorem~\ref{thm: tripleslist}).  In~\cite{NonMinimalTriples} the authors and M. Kumbhat characterize the non-minimal DSF triples; thus Theorem~\ref{thm: tripleslist} completes the characterization of all DSF triples.

The study of DSF sets is motivated by several well-known classes of graphs.  The sets $\{K_2\}$ and $\{2K_1\}$ are the forbidden sets for the edgeless and the complete graphs, respectively, which have trivial degree sequence characterizations. More interesting, $\{2K_2, C_4, P_4\}$, $\{2K_2, C_4, C_5\}$, and $\{2K_2,C_4\}$ are sets of forbidden induced subgraphs for the threshold graphs~\cite{ChvatalHammer73}, split graphs~\cite{FoldesHammer76}, and pseudo-split graphs~\cite{MaffrayPreissmann94}, respectively; each of these classes has a degree sequence characterization (\cite{HammerEtAl78}, \cite{HammerSimeone81,Tyshkevich80,TyshkevichEtAl81}, \cite{MaffrayPreissmann94}). Matroidal and matrogenic graphs also have characterizations in terms of both degree sequences and forbidden subgraphs~\cite{FoldesHammerMatroid,MarchioroEtAl84,Peled77,Tyshkevich84}.

This paper is organized as follows.  Section~\ref{sec: minimal properties} presents the new properties of minimal DSF sets, while Section~\ref{sec: the search} presents additional properties and develops the computer algorithm for searching for minimal DSF triples.  The last two subsections of Section~\ref{sec: the search} contain the results of the computer search.  Finally, we conclude with several open questions about DSF sets in Section~\ref{sec: conclusion}.


\section{Minimal DSF sets} \label{sec: minimal properties}

In this section we establish fundamental properties of minimal DSF sets. We show that every graph belongs to a minimal DSF set, and we determine bounds on the numbers and sizes of these sets and the graphs they contain.

We begin with some definitions. Let $V(G)$ and $E(G)$ denote the vertex and edge sets, respectively, of a graph $G$. The \emph{order} of $G$ is defined as $|V(G)|$. For $W \subseteq V(G)$, we use $G[W]$ to denote the induced subgraph of $G$ with vertex set $W$. The complement of a graph $G$ is denoted by $\overline{G}$. We write the disjoint union of graphs $G$ and $H$ as $G+H$, and we indicate the disjoint union of $m$ copies of $G$ by $mG$. We use $K_n$, $P_n$, and $C_n$, respectively, to denote the complete graph, path, and cycle on $n$ vertices, and we use $K_{m,n}$ to denote the complete bipartite graph with partite sets of sizes $m$ and $n$. The \emph{paw} is the complement of $P_3+K_1$. The \emph{4-pan} is the graph obtained by adding a pendant vertex to a 4-cycle; the \emph{co-4-pan} is the complement of the 4-pan. The \emph{house} is the complement of $P_5$. We illustrate these named graphs in Figure~\ref{fig: small graphs}.
\begin{figure}
\centering
\includegraphics{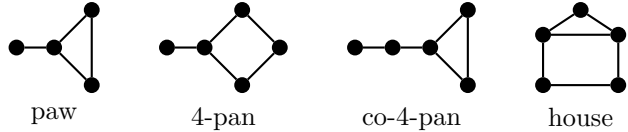}
\caption{Small named graphs.}
\label{fig: small graphs}
\end{figure}

Given vertices $a,b,c,d$ in a graph $G$ with $ab,cd \in E(G)$ and $ad,bc \notin E(G)$, a \emph{2-switch} is the operation of deleting edges $ab$ and $cd$ from $G$ and adding edges $ad$ and $bc$. As shown by Fulkerson, Hoffman, and McAndrew~\cite{FulkersonEtAl65}, two graphs have the same degree sequence if and only if one can be obtained from the other via a sequence of 2-switches.

To show that a set $\F$ of graphs is not DSF, it suffices to produce two graphs $H$ and $H'$ having the same degree sequence in which $H$ contains an element of $\F$ as an induced subgraph while $H'$ is $\F$-free. We call such a pair $(H,H')$ an \emph{$\F$-breaking pair}. In deciding whether a set $\F$ is DSF, one does not have to look far for $\F$-breaking pairs, if they exist.

\begin{lem} \label{lem: at most 2 more vtcs}
If $\F$ is a set of graphs that is not DSF, then there exists an $\F$-breaking pair $(H,H')$ such that $H$ and $H'$ each have at most $|V(G)| + 2$ vertices, where $G$ is a graph in $\F$ having the most vertices.
\end{lem}
\begin{proof}
Since $\F$ is not DSF, by definition there exists an $\F$-breaking pair $(J, J')$ of graphs. Since $J$ and $J'$ have the same degree sequence, there exists a sequence $J = J_0, \dots, J_k = J'$ of graphs in which $J_i$ is obtained via a 2-switch on $J_{i-1}$ for $i\in\{1,\dots, k\}$. Define $\ell$ to be the largest index such that $J_\ell$ induces an element $G$ of $\F$, so that $(J_\ell, J_{\ell+1})$ is an $\F$-breaking pair. Let $V$ denote the vertex set of an induced copy of $G$ in $J_\ell$, and let $W$ denote the set of four vertices involved in the 2-switch transforming $J_\ell$ into $J_{\ell+1}$. Since $G$ is not induced on $V$ in $J_{\ell+1}$, the 2-switch performed must add an edge to or delete an edge from $J_\ell[V]$; hence $|W\cap V| \geq 2$ and $|V \cup W| \leq |V|+2$. Thus $(J_\ell[V\cup W], J_{\ell+1}[V \cup W])$ is an $\F$-breaking pair on at most $|V(G)| + 2$ vertices. Taking $G$ to be a graph in $\F$ with the most vertices yields the result.
\end{proof}

Given a finite set $\F$ of graphs, Lemma~\ref{lem: at most 2 more vtcs} implies that determining if $\F$ is a DSF set can be done in finite time, as there are only a finite number of pairs $(H,H')$ to check for $\F$-breaking pairs.

When we restrict our attention to \emph{minimal} DSF sets, Lemma~\ref{lem: at most 2 more vtcs} imposes restrictions on ``gaps'' in the sizes of vertex sets.

\begin{lem} \label{lem: min sizes differ by at most 2}
Let $\F = \{F_1, \dots, F_k\}$, where the graphs in $\F$ are indexed in order of the sizes of their vertex sets, from smallest to largest. If $\F$ is a minimal DSF set, then $|V(F_{i+1})|-|V(F_i)|\leq 2$ for all $i \in \{1, \dots , k-1\}$.
\end{lem}
\begin{proof}
Let $\F$ be a minimal DSF set with elements indexed as described. For any $i \in \{1, \dots, k-1\}$, let $\G$ be the set $\{F_1, \dots, F_i\}$. Since $\F$ is minimal, $\G$ is not DSF. By Lemma~\ref{lem: at most 2 more vtcs} there exists a $\G$-breaking pair $(H,H')$ such that $|V(H')| \leq |V(F_i)| + 2$. Since $\F$ is DSF, $H'$ induces $F_j$ for some $j > i$. Thus \[|V(F_{i+1})| - |V(F_i)| \leq |V(F_j)| - |V(F_i)| \leq |V(H')| - |V(F_i)| \leq 2,\] as claimed.
\end{proof}

A similar result applies to the sizes of edge sets.

\begin{lem}~\label{lem: min edge sizes differ by}
Let $\F = \{F'_1, \dots, F'_k\}$, where the graphs in $\F$ are indexed in order of the sizes of their edge sets, from smallest to largest. If $\F$ is a minimal DSF set, then $|E(F'_{i+1})| \leq \max\{|E(F'_j)| + 2|V(F'_j)|\;:\;j \leq i\}$ for all $i \in \{1, \dots, k-1\}$.
\end{lem}
\begin{proof}
Let $\F$ be a minimal DSF set with elements indexed as described. For any $i \in \{1, \dots, k-1\}$, let $\G'$ be the set $\{F'_1, \dots, F'_i\}$. Since $\F$ is minimal, $\G'$ is not DSF. As in the proof of Lemma~\ref{lem: at most 2 more vtcs}, there exists a $\G'$-breaking pair $(H,H')$ such that every vertex of $H$ not belonging to a chosen induced copy of some $F'_j$ in $\G'$ belongs to one of the edges involved in the 2-switch that changes $H$ into $H'$. There are at most two such vertices, and these vertices are contained in all edges of $H$ that do not belong to the chosen copy of $F'_j$. Thus $H$ contains at most $|E(F'_j)| + 2|V(F'_j)|$ edges (any vertex not in the copy of $F'_j$ is involved in the 2-switch changing $H$ into $H'$ and hence is not a dominating vertex in $H$). The graph $H'$ has the same number of edges as $H$, and since $\F$ is DSF, $H'$ induces $F'_\ell$ for some $\ell > i$. It follows that \[|E(F'_{i+1})| \leq |E(F'_\ell)|\leq |E(H')|\leq |E(F'_j)| + 2|V(F'_j)|,\]
which completes the proof.
\end{proof}
The previous result and induction yield the following.
\begin{cor}\label{cor: min edge sizes differ by, simpler}
Let $\F = \{F'_1, \dots, F'_k\}$, where the graphs in $\F$ are indexed in order of the sizes of their edge sets, from smallest to largest. If $\F$ is a minimal DSF set, then $|E(F'_{i+1})| \leq |E(F'_1)|+2\sum_{j=1}^i |V(F'_j)|$ for all $i \in \{2, \dots, k\}$.
\end{cor}

Thus minimal DSF sets with few elements contain graphs that do not differ greatly in the numbers of vertices and edges they contain. We now show that small minimal DSF sets must therefore contain small graphs. We first recall a result from~\cite{BarrusEtAl08}.

\begin{thm}[\cite{BarrusEtAl08}] \label{thm: forest of stars}
Every DSF set must contain a forest in which each component is a star, as well as the complement of a forest of stars.
\end{thm}

\begin{thm}~\label{thm: bound on min DSF size}
If $\F$ is a minimal DSF set of $k$ graphs, then the number of vertices in any element of $\F$ is at most \[4k - \frac{3}{2} + \sqrt{8k^2 +2k - \frac{31}{4}};\] hence there are finitely many minimal DSF $k$-sets.
\end{thm}
\begin{proof}
Let $\F$ be a minimal DSF set of $k$ graphs, and denote the graphs of $\F$ both by $\{F_1, \dots, F_k\}$ and by $\{F'_1 , \dots, F'_k\}$, where the graphs are indexed in ascending order of the sizes of their vertex sets and edge sets, respectively. Let $n_1 = |V(F_1)|$. As a consequence of Lemma~\ref{lem: min sizes differ by at most 2}, we observe that $|V(F_k)| \leq n_1 + 2(k - 1)$. Corollary~\ref{cor: min edge sizes differ by, simpler} and Lemma~\ref{lem: min sizes differ by at most 2} imply that
\begin{align}
\nonumber |E(F'_k)| - |E(F'_1)| & \leq 2\sum_{i=1}^{k-1} |V(F'_i)|\\
\nonumber &\leq 2\sum_{i=2}^k |V(F_i)|\\
\nonumber &\leq 2 \sum_{i=2}^k (n_1+2(i-1))\\
&= 2(k-1)n_1 + 2k^2 - 2k. \label{eq: edge difference bound}
\end{align}
We now find a lower bound on the difference in the number of edges between any two graphs in the set $\F$. By Theorem~\ref{thm: forest of stars}, some element $F_s$ of $\F$ is a forest of stars, while some element $F_c$ of $\F$ is the complement of a forest of stars. The complement of a forest with $n$ vertices has at least $n(n-1)/2-(n-1) = (n-2)(n-1)/2$ edges, so we find
\begin{align*}
|E(F'_k)| - |E(F'_1)| &\geq |E(F_c)| - |E(F_s)|\\
&\geq \frac{(n_1-2)(n_1-1)}{2}-(|V(F_k)|-1)\\
&\geq \frac{(n_1-2)(n_1-1)}{2} - (n_1 + 2(k-1)-1).
\end{align*}
Combining this inequality and the one in~\eqref{eq: edge difference bound}, we have \[\frac{(n_1-2)(n_1-1)}{2} - (n_1 + 2(k-1) - 1) \leq 2(k-1)n_1 + 2k^2-2k,\] which reduces to
\begin{equation}\label{eq: bound on n1}
n_1 \leq 2k + \frac{1}{2} + \sqrt{8k^2 +2k - \frac{31}{4}}.
\end{equation}
Since $|V(F_k)| \leq n_1 + 2(k-1)$, the result follows.
\end{proof}

We will use the following bounds for minimal DSF triples when we search for them in Section~\ref{sec: the search}.

\begin{cor} \label{cor: bounds when k is 3}
Suppose $\F=\{F_1,F_2,F_3\}$, where the $F_i$ are indexed in order of the sizes of their vertex sets, from smallest to largest. If $\F$ is a minimal DSF triple, then $|V(F_1)| \leq 14$, $|V(F_2)| \leq 16$, and $|V(F_3)|\leq 18$.
\end{cor}
\begin{proof}
Substituting $k=3$ in~\eqref{eq: bound on n1} yields $|V(F_1)| \leq 14$; the other statements follow from Lemma~\ref{lem: min sizes differ by at most 2}.
\end{proof}

The bound in Theorem~\ref{thm: bound on min DSF size} does not appear to be tight.  Consider the theorem in \cite{BarrusEtAl08} characterizing the DSF sets of order at most 2.
\begin{thm}[\cite{BarrusEtAl08}]\label{thm: pairs list}
A set $\F$ of at most two graphs is DSF if and only if it is one of the following:
\begin{itemize}
\item[\textup{(1)}] a set containing one of $K_1$, $K_2$, or $2K_1$;
\item[\textup{(2)}] one of $\{P_3,K_3\}$, $\{P_3,K_3 + K_1\}$, $\{P_3,K_3 + K_2\}$, $\{P_3, 2K_2\}$, $\{P_3,K_2 + K_1\}$, $\{K_2 + K_1, 3K_1\}$, $\{K_2 + K_1,K_{1,3}\}$, $\{K_2 + K_1,K_{2,3}\}$, $\{K_2 + K_1,C_4\}$, $\{K_3, 3K_1\}$, or $\{2K_2,C_4\}$.
\end{itemize}
\end{thm}
When $k=1$ or $k=2$, Theorem~\ref{thm: bound on min DSF size} states that the largest graph in a minimal DSF singleton or pair has at most four or at most eleven vertices, respectively. Theorem~\ref{thm: pairs list} shows that the largest graph in any DSF singleton has two vertices, and the largest graph in any minimal DSF pair has five vertices. Furthermore, Theorem~\ref{thm: pairs list} and, as we shall see shortly, Theorem~\ref{thm: tripleslist} show that most known minimal DSF sets do not vary as widely in the orders of their elements as Lemma~\ref{lem: min sizes differ by at most 2} allows them to, suggesting that a better bound may be possible.

We remark that Lemmas~\ref{lem: min sizes differ by at most 2} and \ref{lem: min edge sizes differ by} and Theorem~\ref{thm: bound on min DSF size} highlight a major difference between minimal and non-minimal DSF sets. While the numbers, orders, and edge set sizes of graphs in a minimal DSF set satisfy fairly stringent relationships, Theorem~\ref{thm: pairs list}(1) and results in~\cite{NonMinimalTriples} show that this is not necessarily true for non-minimal DSF sets. For example, $\{2K_2,C_4,K_n\}$ is a non-minimal DSF triple for any $n\ge 1$.  If $\F$ and $\F^*$ are DSF and $\F$ properly contains $\F^*$, then the graphs of $\F \setminus \F^*$ can differ from the graphs in $\F^*$ in terms of the sizes of their vertex and edge sets by arbitrarily much.

Although Theorem~\ref{thm: bound on min DSF size} states that there are finitely many minimal DSF sets of a given size, there are infinitely many minimal DSF sets: the set $\G_n$ of all graphs on $n$ vertices is clearly DSF and must contain a finite minimal DSF subset. Thus Theorem~\ref{thm: bound on min DSF size} also implies that there are arbitrarily large finite DSF sets.

We next show that every finite graph belongs to a minimal DSF set.  We will use the following definition.

\begin{dfn}
Let $\G$ be a set of graphs in which no element is induced in another. Consider the set of all graphs that have the same degree sequence as a graph that induces an element of $\G$, and let $\D(\G)$ be the minimal elements of this set under induced subgraph containment.  When $\G$ contains a single graph $G$, we write $\D(G)$ for convenience.
\end{dfn}

Note that $\D(\G)$ contains $\G$.  As a simple example, we compute $\D(K_3)$.

\begin{exa}\label{exa: K_3-sieve}
We show that $\D(K_3) = \{K_3,P_5, \text{4-pan}, K_{2,3}, K_2+C_4, K_2+C_5, P_3+P_4, 3P_3\}$.  For each of the graphs listed one may easily construct a $K_3$-inducing graph having the same degree sequence; this property does not hold for any of the proper induced subgraphs of a graph in the list. Furthermore, no graph in the list is induced in any other graph in the list. Hence $\{K_3,P_5, \text{4-pan}, K_{2,3}, K_2+C_4, K_2+C_5, P_3+P_4, 3P_3\} \subseteq \D(K_3)$.

Suppose that $G$ contains no element of the list as an induced subgraph. We show that the degree sequence of $G$ has no $K_3$-inducing realizations; in other words, the list above contains $\D(K_3)$. We begin by supposing that $G$ induces a cycle. Such a cycle must have length 4 or 5, since $G$ is $\{K_3, P_5\}$-free. Since $G$ is $\{K_2+C_4,K_2+C_5\}$-free, the component of $G$ containing the cycle is the only nontrivial component. If the cycle is a 4-cycle, then the component contains no other vertices, since otherwise some vertex on the cycle would have to be adjacent to a cycle vertex, and $G$ is $\{K_3,\text{4-pan},K_{2,3}\}$-free.  Likewise, if the cycle is a 5-cycle, then the component has no other vertices, since $G$ is $\{K_3,P_5,\text{4-pan}\}$-free. Thus if $G$ induces a cycle, it has the form $C_k + nK_1$, where $k \in \{4,5\}$. In this case $G$ is the unique realization of its degree sequence; hence $d(G)$ has no $K_3$-inducing realizations.

If $G$ induces no cycle, then $G$ is a forest. Since $G$ is $P_5$-free, each component has diameter at most 3. Suppose that $G$ has component with diameter 3. Since $G$ is $P_3+P_4$-free, every other component of $G$ is $P_3$-free and hence complete; since $G$ is a forest we conclude that $G$ is the disjoint union of a diameter-3 tree (often called a ``double star'') with some number of components with orders 1 or 2.

Suppose instead that $G$ only has components with diameters at most 2; in this case, $G$ is a disjoint union of stars. Since $G$ is $3P_3$-free, $G$ has at most two components with order greater than 2. We conclude that $G$ has the form $K_{1,a} + K_{1,b} + mK_2 + nK_1$ or is a proper induced subgraph of such a graph.

Note now that any 2-switch on the union of a double star with components of orders 1 or 2 either preserves the isomorphism class or changes it to a graph having at most two $P_3$-inducing star components and all other components having order 1 or 2. Any 2-switch on a graph of this latter type either preserves the isomorphism class or changes it to a double star plus components or orders 1 or 2. Thus in every case the degree sequence of $G$ has no $K_3$-inducing realization.
\end{exa}

The set $\D(\G)$ is closely tied to DSF sets containing $\G$. If $\F$ is a DSF set containing $\G$ as a subset, then every element of $\D(\G)$ must induce at least one element of $\F$. Furthermore, $\D(\G)$ is minimal among DSF sets containing $\G$:

\begin{lem} \label{lem: S is DSF and S_0-minimal}
The set $\D(\G)$ is DSF, and no proper subset of $\D(\G)$ that contains $\G$ is DSF.
\end{lem}
\begin{proof}
Let $H$ and $H'$ be two graphs with the same degree sequence, and suppose that $H$ induces an element $F$ of $\D(\G)$. By definition, $F$ has the same degree sequence as a graph $F'$ that induces an element of $\G$. There exists a sequence of 2-switches that changes $F$ into $F'$; applying these 2-switches to the induced copy of $F$ in $H$ changes $H$ into a graph $H''$. Since $H'$ has the same degree sequence as $H''$ and $H''$ induces $F'$, by definition $H'$ induces an element of $\D(\G)$. Hence no $\D(\G)$-breaking pair exists; thus $\D(\G)$ is DSF.

Suppose now that $\D'$ is a proper subset of $\D(\G)$ that contains $\G$ and is DSF. Let $F$ be a graph in $\D(\G)\setminus\D'$. Note that $F$ is $\D'$-free, by the assumption of minimality under induced subgraph containment in the definition of $\D(\G)$. However, by definition $F$ has the same degree sequence as a graph $F'$ that induces an element of $\G$. Since $\D'$ contains $\G$, we see that $(F', F)$ is a $\D'$-breaking pair, a contradiction.
\end{proof}

\begin{thm} \label{thm: every graph in minimal}
Every graph belongs to a minimal DSF set.
\end{thm}
\begin{proof}
Let $G$ be an arbitrary graph, and let $\G$ be the set of all graphs having the same degree sequence as $G$. By Lemma~\ref{lem: S is DSF and S_0-minimal}, $\D(\G)$ is a DSF set, and any proper subset that is DSF must omit an element of $\G$. Suppose that $\D'$ is a proper subset of $\D(\G)$ that is DSF and omits a graph $F$ in $\G$. In order to prevent a $\D'$-breaking pair, $\D'$ cannot contain any element of $\G$. However, consider an element $H$ of $\D'$ with minimum order. Because $H$ belongs to $\D(\G)$, it has the same degree sequence as a graph $H'$ that induces an element of $\G$. As such, $H'$ does not belong to $\D(\G)$ (and hence $H' \notin \D'$), since by definition no element of $\D(\G)$ contains another element of $\D(\G)$ as an induced subgraph. Nor does any proper induced subgraph of $H'$ belongs to $\D'$, since $H$ was assumed to have minimum order in $\D'$. Thus $(H,H')$ is a $\D'$-breaking pair, a contradiction. Hence no proper subset of $\D(\G)$ omitting an element of $\G$ can be DSF, and $\D(\G)$ is therefore a minimal DSF set.
\end{proof}

We conclude this section by using the sets $\D(\G)$ to give a new characterization of DSF sets.

\begin{thm} \label{thm: DSF iff D condition holds}
Let $\F$ be a set of graphs in which no element is induced in another. The following are equivalent:
\begin{enumerate}
 \item[\textup{1.}] $\F$ is DSF.
 \item[\textup{2.}] For each element $G$ of $\F$, every element of $\D(G)$ contains an element of $\F$ as an induced subgraph.
\item[\textup{3.}] $\D(\F)=\F$.
 \end{enumerate}
\end{thm}
\begin{proof}
$1.\Rightarrow 2.$ Let $G$ be an element of $\F$. By definition, every element of $\D(G)$ other than $G$ itself belongs to a $G$-breaking pair, so since $\F$ is DSF, every element of $\D(G)$ contains an element of $\F$ as an induced subgraph.

$2.\Rightarrow 3.$ As $\F\subseteq \D(\F)$, we need only show that $\D(\F)\subseteq \F$.  Suppose $H\in\D(\F)$.  By definition, $H$ has the same degree sequence as some graph $H'$ that induces a graph $G$ of $\F$. No proper induced subgraph of $H$ has this property, so by definition, $H\in \D(G)$. By assumption, $H$ contains an element of $\F$ as an induced subgraph. Since $H$ is in $\D(\F)$, no proper induced subgraph of $H$ is an element of $\F$, so in fact $H$ is an element of $\F$.

$3.\Rightarrow 1.$  Let $(H,H')$ be an $\F$-breaking pair. By definition of $\D(\F)$, $H'$ must induce an element of $\D(\F)$, but since $\D(\F)=\F$, $H'$ induces an element of $\F$ and hence $(H,H')$ is not an $\F$-breaking pair.  Since no $\F$-breaking pair exists, $\F$ is DSF.
\end{proof}


\section{Minimal DSF triples} \label{sec: the search}

We now use properties of minimal DSF sets to determine all minimal DSF triples.  By Corollary~\ref{cor: bounds when k is 3}, this is a finite problem: only triples of graphs with at most $18$ vertices need be considered.  However, examining all such triples is infeasible, since there are approximately $1.79\times10^{30}$ graphs on $18$ vertices (see \cite{oeis}).  In this section we present a practical algorithm that identifies these sets by ruling out all non-suitable triples. We present the complete list of minimal DSF triples in Theorem~\ref{thm: tripleslist}.

Suppose that $\F=\{F_1,F_2,F_3\}$ is a minimal DSF triple.  By Theorem~\ref{thm: forest of stars}, every DSF set must contain a graph that is a forest of stars and a graph (possibly the same graph) that is a complement of a forest of stars.  Additionally, we show here that every DSF set must also contain a complete bipartite graph and a disjoint union of cliques where at most one clique has more than two vertices, as well as complements from these two families.

We construct potential minimal DSF triples $\F$ by ensuring that each of the six graph classes above is satisfied by one of the graphs in $\F$.  Very small graphs can be in several of the graph classes at once, so we first handle the case where one graph of $\F$ has at most three vertices.  For the remaining case, we may assume that each graph of $\F$ has at least four vertices.  As there are six classes that must be represented in each DSF triple, some graph in the triple must be a member of more than one class, giving additional restrictions.  Potential triples that violate the vertex and edge bounds of Lemma~\ref{lem: min sizes differ by at most 2} or Corollary~\ref{cor: bounds when k is 3} are then eliminated, and the remaining triples are then examined with a brute force search for $\F$-breaking pairs.  Triples for which no $\F$-breaking pairs are found are then DSF sets.

\subsection{Graph classes appearing in every DSF set}

In~\cite{BarrusEtAl08} the authors and M. Kumbhat showed that every DSF set must contain a representative from certain graph classes.  As we recalled in Theorem~\ref{thm: forest of stars}, one of these graph classes is the class of forests of stars.  The following lemma then implies that every DSF set must also then contain the complement of a forest of stars.

\begin{lem}[\cite{BarrusEtAl08}]\label{lem: complements}
Given any collection $\F$ of graphs, let $\F^c$ denote the collection of graphs which are complements of those in $\F$.  Then $\F$ is DSF if and only if $\F^c$ is DSF.
\end{lem}

We here extend the list of graph classes from which every DSF set must have a representative.
In~\cite{BarrusEtAl08} we see that every DSF set contains a disjoint union of cliques; this is strengthened in the following theorem.

\begin{thm}\label{thm: DSF contains disjoint unions}
Every DSF set contains a complete bipartite graph and a disjoint union of cliques with at most one clique with more than two vertices.  Moreover, every DSF set must also contain a disjoint union of at most two complete graphs and a complete multipartite graph where at most one partite set has order greater than two.
\end{thm}
\begin{proof}
Let $\F$ be a DSF set, and let $F$ be any graph of $\F$. Let $m$ be the number of edges in the complement $\overline{F}$. We form a graph $H$ by starting with the disjoint union of $F$ and the complete multipartite graph $M$ having $m$ partite sets of size 2. We bijectively assign each partite set $\{a,b\}$ in $M$ to a pair $\{c,d\}$ of nonadjacent vertices in $F$ and add the edges $ac$ and $bd$ to $H$. In the resulting graph, each vertex belonging to $F$ now has degree $|V(F)|-1$, and every vertex of $M$ has degree $|V(M)|-1$. The degree sequence of $H$ thus matches that of $K_{|V(F)|} + K_{|V(M)|}$, which we denote by $H'$. Since $\F$ is DSF, $(H,H')$ cannot be an $\F$-breaking pair, so $\F$ contains a graph induced by $H'$. Thus $\F$ contains a disjoint union of at most two complete graphs.

Applying a similar construction again to $F$, but with $M$ an edgeless graph with $2m$ vertices instead, we construct an $F$-inducing graph $H$ having the same degree sequence as $H'=K_{|V(F)|}+mK_2$; hence $\F$ must also contain a disjoint union of complete graphs where at most one component has order greater than two.

Finally, by Lemma~\ref{lem: complements} the set $\F^c$ is also DSF and hence must contain disjoint unions of complete graphs of both of the types described. The set $\F$ contains the complements of these graphs and hence a complete bipartite graph and a complete multipartite graph where at most one partite set has order greater than two.
\end{proof}

Note that such results were proved in \cite{BarrusEtAl08} in the framework of monotone parameters.  This framework also works for the previous result, where edit distance (the number of edges that must be added or removed to change the graph into, for instance, a complete bipartite graph) is the monotone parameter.  We have presented a direct proof here for simplicity.

\subsection{Minimal DSF triples containing a small graph}

We first handle the cases where a minimal DSF triple contains a graph with at most three vertices.

\begin{lem}\label{lem: small graphs not in min DSF triples}
\begin{enumerate}
\item[\textup{(a)}] No minimal DSF triple contains a graph having order 1 or 2.
\item[\textup{(b)}] No minimal DSF triple contains $P_3$ or $K_1+K_2$.
\end{enumerate}
\end{lem}
\begin{proof}
By Theorem~\ref{thm: pairs list}, any graph of order 1 or 2 comprises a DSF set; hence a DSF triple containing such a graph cannot be minimal. Let $\F$ be a DSF set containing $P_3$. Since $(P_5,K_3+K_2)$ is not an $\F$-breaking pair, $\F$ must contain an induced subgraph $G$ of $K_3+K_2$. However, no matter what $G$ is, $\{P_3,G\}$ appears in Theorem~\ref{thm: pairs list}, so $\F$ cannot be a minimal DSF triple. By Theorem~\ref{thm: pairs list}(1), the complement of a minimal DSF triple containing $K_1+K_2$ would be a minimal DSF triple containing $P_3$, so no such triple exists.
\end{proof}

We now characterize the minimal DSF triples containing $K_3$ or $3K_1$.

\begin{lem} \label{lem: triples with K3}
If $\F$ is a minimal DSF triple containing $K_3$, then $\F$ is one of $\{K_3, P_3+K_1, K_{1,3}\}$, $\{K_3, P_3+K_1, C_4\}$, $\{K_3, P_3+K_1, K_{2,3}\}$, and $\{K_3, 2K_2, K_{1,3}\}$.
\end{lem}
\begin{proof}
Let $\F=\{K_3,F,G\}$, and suppose that $\F$ is a minimal DSF set. Since $K_3$ forms a DSF pair with $3K_1$, it follows from Lemma~\ref{lem: small graphs not in min DSF triples} that $|F_2|, |F_3| \geq 4$. By Example~\ref{exa: K_3-sieve}, in order to prevent $\F$-breaking pairs, each element of $\{P_5, \text{4-pan}, K_{2,3}, K_2+C_4, K_2+C_5, P_3+P_4, 3P_3\}$ must induce an element of $\F$. Thus, $P_5$ must induce either $F$ or $G$; without loss of generality we may assume that $P_5$ induces $F$. Since $K_{2,3}$ induces $F$ or $G$, but $P_5$ and $K_{2,3}$ have no common induced subgraphs with at least four vertices, we conclude that $K_{2,3}$ induces $G$. Similarly, since $K_{2,3}$ and $3P_3$ have no common induced subgraphs with at least four vertices, $3P_3$ must induce $F$. Thus $F$ is a graph on at least four vertices induced in $P_5$ and $3P_3$; the only such graphs are $P_3+K_1$ and $2K_2$.

If $F=P_3+K_1$, then by letting $G$ vary over induced subgraphs of $K_{2,3}$ with at least four vertices, we obtain the following as possibilities for $\F$: $\{K_3, P_3+K_1, K_{1,3}\}$, $\{K_3, P_3+K_1, C_4\}$, $\{K_3, P_3+K_1, K_{2,3}\}$.

If $F=2K_2$, then $G$ must be an induced subgraph of both $K_{2,3}$ and the 4-pan graph; thus $G$ is $K_{1,3}$ or $C_4$. Since $\{2K_2,C_4\}$ is a DSF pair, however, the only possibility for a minimal DSF triple in this case is $\{K_3, 2K_2, K_{1,3}\}$.
\end{proof}

\begin{cor} \label{cor: triples with K2K1}
If $\F$ is a minimal DSF triple containing $3K_1$, then $\F$ is one of $\{3K_1, \text{paw}, K_1+K_3\}$, $\{3K_1, \text{paw}, 2K_2\}$, $\{3K_1, \text{paw}, K_2+K_3\}$, and $\{3K_1, C_4, K_1+K_3\}$.
\end{cor}

As we shall see in Section~\ref{sec: tripleslist}, the sets in Lemma~\ref{lem: triples with K3} and Corollary~\ref{cor: triples with K2K1} are indeed minimal DSF triples.

\subsection{Minimal DSF triples of large graphs}

We now search for minimal DSF triples $\F=\{F_1,F_2,F_3\}$ where each graph has at least four vertices.  Our search algorithm consists of two phases: in the first phase, we construct a list of candidate triples that satisfy the six necessary graph classes.  In the second phase, we then check each of those triples to determine if it is a minimal DSF triple.

\subsubsection{Phase I: constructing potential triples}

By Theorems~\ref{thm: forest of stars} and \ref{thm: DSF contains disjoint unions}, a minimal DSF triple $\F$ must contain a representative from each of the following graph classes:
\begin{align*}
\B & =\{\textrm{complete bipartite graphs}\};\\
\K & =\{\textrm{disjoint unions of cliques, at most one with more than two vertices}\};\\
\S & =\{\textrm{forests of stars}\}.
\end{align*}
Let $\B^{c}$, $\K^{c}$, and $\S^{c}$ denote the sets of complements of graphs in these respective classes.  Theorems~\ref{thm: forest of stars} and \ref{thm: DSF contains disjoint unions}, also imply that $\F$ also contains a representative from $\B^{c}$, $\K^{c}$, and $\S^{c}$.  Note that graphs in $\B^c$ are disjoint unions of at most two complete graphs, and graphs in $\K^c$ are complete multipartite graphs.

 As $\F$ only contains three graphs, a graph in $\F$ must lie in the intersection of more than one of the six graph classes.  In Table~\ref{tab: pairwise intersections}, we describe the pairwise intersections of these families, listing only graphs with at least four vertices. The first two columns of each row show which families are involved in the intersection. The first column shows a typical element of the first set, indicating the notation used for describing a graph of the class. The third column shows the restrictions imposed on these parameters by membership in the second graph class.

\begin{table}
\caption{\label{tab: pairwise intersections}Lists the pairwise intersections of the graph classes, considering only graphs with at least four vertices.}
\begin{center}
\begin{tabular}{|c|c|c|}
\hline
First class&
Second&
Graphs in the intersection\tabularnewline
\hline
\hline
\multirow{5}{*}{$K_{a_{1},a_{2}}\in\B$, $a_{1}\le a_{2}$}&
$\B^{c}$&
none\tabularnewline
\cline{2-3}&
$\K$&
$a_{1}=0, a_2 \geq 4$\tabularnewline
\cline{2-3}&
$\K^{c}$&
$a_{1}\le2$\tabularnewline
\cline{2-3}&
$\S$&
$a_{1}\le1$\tabularnewline
\cline{2-3}&
$\S^{c}$&
$a_{1}=a_{2}=2$\tabularnewline
\hline
\multirow{4}{*}{$K_{b_{1}}+K_{b_{2}}\in\B^{c}$, $b_{1}\le b_{2}$}&
$\K$&
$b_{1}\le2$\tabularnewline
\cline{2-3}&
$\K^{c}$&
$b_{1}=0$, $b_2 \geq 4$\tabularnewline
\cline{2-3}&
$\S$&
$b_{1}=b_{2}=2$\tabularnewline
\cline{2-3}&
$\S^{c}$&
$b_{1}\le1$\tabularnewline
\hline
\multirow{3}{*}{$K_{c_{3}}+c_{2}K_{2}+c_{1}K_{1}\in\K$,
$c_{3}=0$ or $c_{3}\ge3$}&
$\K^{c}$&
($c_{3}=c_{2}=0$) or ($c_{2}=c_{1}=0$)\tabularnewline
\cline{2-3}&
$\S$&
$c_{3}=0$\tabularnewline
\cline{2-3}&
$\S^{c}$&
$c_{1}\le1$ and $c_{2}=0$\tabularnewline
\hline
\multirow{2}{*}{$c_{3}K_{1}\vee(\overline{c_{2}K_{2}})\vee K_{c_{1}}\in\K^{c}$,
$c_{3}=0$ or $c_{3}\ge3$}&
$\S$&
$c_{1}\le1$ and $c_{2}=0$\tabularnewline
\cline{2-3}&
$\S^{c}$&
$c_{3}=0$\tabularnewline
\hline
$\S$&
$\S^{c}$&
none\tabularnewline
\hline
\end{tabular}
\end{center}
\end{table}

Note that the intersection of $\B$ and $\B^c$ is empty for graphs with at least four vertices. Additionally, complete bipartite graphs and their complements are parametrized by only two values.  Thus, for our minimal DSF triple $\F$, we choose $F_1\in\B$ and $F_2\in \B^c$.  The specific choices of $F_1$ and $F_2$ may also be in $\K$, $\K^c$, $\S$, and $\S^c$.  The graph classes that $F_1$ and $F_2$ do not satisfy constrain the choice of $F_3$.

It is possible that $F_1$ and $F_2$ by themselves satisfy all six graphs classes, thus placing no constraints on $F_3$.  The following lemma describes the graphs $F_1$ and $F_2$ when this occurs.

\begin{lem}\label{lem: F1 and F2 satisfy all classes}
Let $F_{1}=K_{a_{1},a_{2}}\in\B$ and $F_{2}=K_{b_{1}}+K_{b_{2}}\in\B^{c}$, where $a_{1}\le a_{2}$ and $b_{1}\le b_{2}$ and both graphs have at least four vertices. The set $\{F_1,F_2\}$ contains an element from each of $\K$, $\K^c$, $\S$, and $\S^c$ if and only if both $a_1$ and $b_1$ are at most $1$ or $a_{1}=a_{2}=b_{1}=b_{2}=2$.
\end{lem}
\begin{proof}
From Table~\ref{tab: pairwise intersections} we see that if $a_1$ and $b_1$ are both at most 1, then $F_1 \in \B \cap \K^c \cap \S$ and $F_2 \in \B^c \cap \K \cap \S^c$. If $a_1=a_2=b_1=b_2$, then $F_1 \in \B \cap \K^c \cap \S^c$ and $F_2 \in \B^c \cap \K \cap \S$.

Conversely, suppose that $\{F_1,F_2\}$ contains an element from each of the four listed classes. The intersection of $\S$ and $\S^{c}$ contains no graph with four or more vertices. Hence, $F_{1}$ belongs to one of $\S$ or $\S^{c}$, and $F_{2}$ belongs to the other. If $F_{1}\in\S$ and $F_{2}\in\S^{c}$, then $a_{1}\le1$ and $b_{1}\le1$, as listed in Table~\ref{tab: pairwise intersections}. If instead $F_{1}\in\S^{c}$ and $F_{2}\in\S$, then $a_{1}=a_{2}=2$ and $b_{1}=b_{2}=2$.
\end{proof}

Note that in the last case, $\{ F_{1},F_{2}\}$ is $\{ C_{4},2K_{2}\}$, which is itself a DSF pair and hence is contained in no minimal DSF triple.  Thus, if $a_1\le 1$ and $b_1 \le 1$, then there is no constraint from the graph classes on the choice of $F_3$.  The following lemmas handle these cases.

\begin{lem}\label{lem: empty graph not DSF}
Suppose that $F_1=K_{a_1,a_2}$ and $F_2=K_{b_1}+K_{b_2}$.
If $a_2\ge4$ and $F_2$ and $F_3$ each have at least $4$ vertices, then $\{ K_{0,a_2},F_2,F_3\}$ is not a DSF set.
By considering complements, if $b_2\ge 4$ and $F_1$ and $F_3$ each have at least $4$ vertices, then $\{F_1,K_{b_2},F_3\}$ is not a DSF set.
\end{lem}
\begin{proof}
Suppose to the contrary that $\{ K_{0,a_2},F_2,F_3\}$ is a minimal DSF set, and let $F_2=K_{b_1}+K_{b_2}$ as in the hypothesis. Consider the following $K_{0,a_{2}}$-breaking pairs: \begin{align*}
H_{1}=P_{5}+(a_{2}-3)K_{1} & \rightrightarrows H_{1}'=K_{2}+K_{3}+(a_{2}-3)K_{1}\\
H_{2}=K_{2,3}+(a_{2}-3)K_{1} & \rightrightarrows H_{2}'=\textrm{house}+(a_{2}-3)K_{1}\\
H_{3}=2P_{3}+(a_{2}-4)K_{1} & \rightrightarrows H_{3}'=P_{2}+P_{4}+(a_{2}-4)K_{1}\end{align*}
Since $\mathcal{F}=\{ K_{0,a_{2}},F_{2},F_{3}\}$ is DSF, for each $i \in \{1,2,3\}$, the graph $H'_i$ induces $F_{2}$ or $F_{3}$.

We first note that $F_2$ cannot be induced in $H'_2$ or $H'_3$. Indeed, $F_2$ is a disjoint union of at most two complete graphs, and the unique such induced subgraphs of $H'_2$ and $H'_3$ are $K_1+K_3$ and $2K_2$, respectively. Suppose $F_2$ is one of these graphs. Since $(\cofourpan,\fourpan)$ is a $\{K_{0,a_2}, F_2\}$-breaking pair, $F_3$ must be induced in the 4-pan. Let $H_4$ and $H'_4$ be the two realizations of $(4,4,3,3,2,2)$ shown in Figure~\ref{fig: lem_3_7}. 
The pair $(H_4,H'_4)$ is also $\{K_{0,a_2},F_2\}$-breaking, so $F_3$ is induced in $H'_4$; thus $F_3$ is $C_4$ or $K_{1,3}$. However, in every case this yields a contradiction: $\{K_{0,a_2},2K_2,K_{1,3}\}$ contains no element of $\mathbb{S}^c$; $\{K_{0,a_2},2K_2,C_4\}$ properly contains $\{2K_2,C_4\}$, a DSF set; and $(K_2+K_3,P_5)$ is a breaking pair for both $\{K_{0,a_2}, K_1+K_3, C_4\}$ and $\{K_{0,a_2}, K_1+K_3, K_{1,3}\}$.

Thus $F_3$ is induced in both $H'_2$ and $H'_3$. This forces $F_3$ to be an induced subgraph of $P_4 + (a_2-3)K_1$.

We now claim that $F_3$ is induced in $H'_1$. Suppose this is not the case. Then $F_2$ is induced in $H_1'$ and must be either $K_1+K_3$ or $K_2+K_3$ (we know that $F_2 \neq 2K_2$, since $F_2$ is not induced in $H'_3$). In either case, $(K_2+K_3,P_5)$ is a $\{K_{0,a_2},F_2\}$-breaking pair, so $F_3$ is induced in $P_5$. Since $F_3$ is also induced in $P_4 + (a_2-3)K_1$, we find that $F_3$ is $P_4$ or $P_3+K_1$. This is a contradiction if $F_2=K_2+K_3$, since $\{K_{0,a_2},K_2+K_3,F_3\}$ contains no element of $\S^c$. Furthermore, the pair $(H_4,H'_4)$ in Figure~\ref{fig: lem_3_7} is $\{K_{0,a_2},K_1+K_3,F_3\}$-breaking whether $F_3$ is $P_4$ or $P_3+K_1$.

Hence $F_3$ is induced in $H'_1$; since it is also induced in $P_4+(a_2-3)K_1$, we find that $F_3$ is induced in $K_2+(a_2-2)K_1$. The classes $\B^c$ and $\S^c$ both contain neither $K_{0,a_2}$ nor $F_3$; thus, $F_2$ belongs to $\B^c \cap \S^c$. As noted in Table~\ref{tab: pairwise intersections}, this implies that $b_1\le 1$. Let $H_5=K_2+K_{b_2}$, and let $H'_5$ be the graph obtained by deleting an edge $uv$ from $K_{b_2}$ and attaching a pendant vertex to each of $u$ and $v$, as shown in Figure~\ref{fig: lem_3_7}.  Since $\alpha(H'_5) < 4$ and $\omega(H_5')<b_2$, the pair $(H_5,H_5')$ is $\{K_{0,a_2},F_2\}$-breaking. Thus $F_3$ is induced in $H_5'$.  Since $F_3$ has at least $4$ vertices, this implies that $F_3=K_2+2K_1$ and $b_2 \geq 4$. Let $H_6$ and $H'_6$ be the graphs with degree sequence $(3,3,3,2,2,1)$ shown in Figure~\ref{fig: lem_3_7}.
Since $H_6$ induces $K_2+2K_1$ but $H'_6$ does not, and $\alpha(H'_6)<4$ and $\omega(H_6')<b_2$, the pair $(H_6,H'_6)$ is $\{F_1,F_2,F_3\}$-breaking, yielding a contradiction.
\end{proof}
\begin{figure}
\centering
\includegraphics{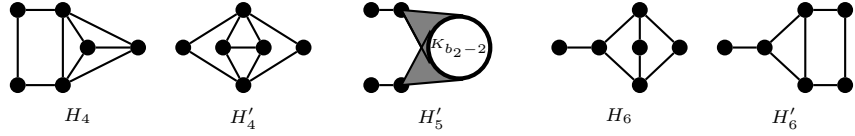}
\caption{Graphs from Lemmas~\ref{lem: empty graph not DSF} and~\ref{lem: star and star complement not DSF}.}
\label{fig: lem_3_7}
\end{figure}

\begin{lem}\label{lem: star and star complement not DSF}
If $a_2\ge3$ and $b_2\ge3$ and $F_3$ has at least $4$ vertices, then $\{ K_{1,a_2},K_1+K_{b_2},F_3\}$ is not a DSF set.
\end{lem}
\begin{proof}
Since $\{ K_{1,a_2},K_1+K_{b_2},F_3\}$ is a minimal DSF set if and only if $\{ K_{1,b_2}, K_1+K_{a_2}, \overline{F_3}\}$ is, for convenience we assume that $3 \le b_2 \le a_2$ and that $\{ K_{1,a_2},K_1+K_{b_2},F_3\}$ is DSF.

Consider the $K_1+K_{b_2}$-breaking pairs $(K_{b_2}+\binom{b_2}{2} K_2,b_2 K_{1,b_2-1})$ and $(K_2+K_{b_2}, H'_5)$, where $H'_5$ is illustrated in Figure~\ref{fig: lem_3_7}. Since the maximum degree of the second graph in both pairs is $b_2-1$ and $b_2\le a_2$, the graph $K_{1,a_2}$ is induced in neither of these second graphs.  Thus, $F_3$ is induced in both $b_2 K_{1,b_2-1}$ and $H'_5$.  Since $b_2 K_{1,b_2-1}$ is triangle-free, $F_3$ is a triangle-free induced subgraph of $H'_5$, which implies that $F_3$ is an induced subgraph of $P_5$.  Since $b_2 K_{1,b_2-1}$ contains no paths of length 3, each component of $F_3$ has diameter at most $2$, and thus $F_3$ is either $2K_2$ or $P_3+K_1$.

Note that in either case $(K_{2,3},\text{house})$ is a $\{K_{1,3},K_1+K_3,F_3\}$-breaking pair, so $a_2 \geq 4$. If $F_3=2K_2$, we then have a contradiction, for $(\cofourpan,\fourpan)$ is $\{K_{1,a_2}, K_1+K_{b_2}, 2K_2\}$-breaking.

Thus $F_3=P_3+K_1$. If $b_2=3$, then the pair $(H_4,H'_4)$ in Figure~\ref{fig: lem_3_7} is $\{K_{1,a_2}, K_1+K_{b_2}, P_3+K_1\}$-breaking, a contradiction. For $b_2 \geq 4$, consider the pair $(H_7,H'_7)$ illustrated in Figure~\ref{fig: lem_3_8}, where $H_7$ is formed by taking two nonadjacent vertices and attaching both to every vertex of $P_3+K_1$, and $H'_7$ is formed by adding three edges to $K_2+K_4$ as shown. Since $\alpha(H'_7)<3$, it is easy to see that $(H_7,H'_7)$ is a $\{K_{1,a_2}, K_1+K_{b_2}, P_3+K_1\}$-breaking pair, a contradiction.
\end{proof}
\begin{figure}
\centering
\includegraphics{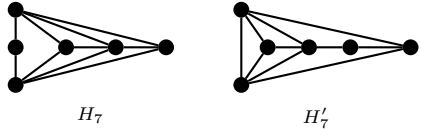}
\caption{Graphs from Lemma~\ref{lem: star and star complement not DSF}.}
\label{fig: lem_3_8}
\end{figure}

Algorithm~\ref{alg:find_potential_triples} presents the first phase of the algorithm in pseudocode.  The output is a list of candidate triples $\F$ that satisfy each of the six graph classes, as well as the bounds of Lemma~\ref{lem: min sizes differ by at most 2} and Corollary~\ref{cor: bounds when k is 3}, but that are not ruled out by Lemmas~\ref{lem: empty graph not DSF} and \ref{lem: star and star complement not DSF}.

\renewcommand{\algorithmicforall}{\textbf{for each}}
\renewcommand{\algorithmiccomment}[1]{\textit{#1}}

\medskip
\begin{alg}\label{alg:find_potential_triples}
Algorithm for Phase I for finding candidate triples.
\medskip
  \begin{algorithmic}[1]
    \FORALL{$F_1=K_{a_1,a_2}\in\B$, where $1 \le a_1 \le a_2$ and $4\le a_1+a_2\le 18$}
      \STATE \COMMENT{We may assume that $1 \le a_1$ by Lemma~\ref{lem: empty graph not DSF}.}
      \STATE Set $n_1 \leftarrow a_1+a_2$
      \FORALL{$F_2=K_{b_1}+K_{b_2}\in\B^c$, where $1 \le b_1 \le b_2$ and $4\le b_1+b_2\le 18$}
        \STATE \COMMENT{We may assume that $1 \le b_1$ by Lemma~\ref{lem: empty graph not DSF}.}
        \STATE Set $n_2 \leftarrow b_1+b_2$ \\\bigskip
        \STATE \COMMENT{Based on $F_1$ and $F_2$ we can eliminate some cases.}
        \IF{($a_1\le 1$ and $b_1\le 1$) or ($a_1=a_2=b_1=b_2=2$)}
          \STATE \COMMENT{By Lemma~\ref{lem: F1 and F2 satisfy all classes}, $F_1$ and $F_2$ already satisfy all six classes. Lemmas~\ref{lem: empty graph not DSF} and \ref{lem: star and star complement not DSF} show that no choices for $F_3$ result in DSF sets, and hence there is no need to check any graphs for $F_3$.}\\
          \STATE \textbf{continue} with the next choice for $F_2$
        \ELSIF{neither $\S$ nor $\S^c$ is satisfied by $F_1$ or $F_2$}
          \STATE \COMMENT{No one graph with at least $4$ vertices can satisfy both $\S$ and $\S^c$, so there is no valid choice for $F_3$.}\\
          \STATE \textbf{continue} with the next choice for $F_2$
        \ENDIF \bigskip

        \FOR{$n_3=4$ to 18}
          \STATE \COMMENT{We will choose an $F_3$ with $n_3$ vertices.}
          \IF{$n_1,n_2,n_3$ do not satisfy Lemma~\ref{lem: min sizes differ by at most 2}}
            \STATE \textbf{continue} with the next choice for $F_3$
          \ENDIF \bigskip

          \STATE
          \COMMENT{We choose $F_3$ to satisfy one of the classes not satisfied by $F_1$ or $F_2$.}
          \IF{$F_1,F_2\notin \K$}
            \FORALL{$F_3=K_{c_{3}}+c_{2}K_{2}+c_{1}K_{1}\in\K$, $c_{3}=0$ or $c_{3}\ge3$, and $c_1+c_2+c_3=n_3$}
              \IF{$F_1,F_2,F_3$ do not satisfy Lemma~\ref{lem: min edge sizes differ by} or Corollary~\ref{cor: bounds when k is 3} on number of edges}
                \STATE \textbf{continue} with the next choice for $F_3$
              \ENDIF
              \IF{$F_1,F_2,F_3$ satisfy all six classes}
                \STATE add $\{F_1,F_2,F_3\}$ to the list of candidate triples
              \ENDIF
            \ENDFOR \COMMENT{ (choosing $F_3\in \K$)} \bigskip

          \ELSIF{$F_1,F_2\notin \K^c$}
            \FORALL{$F_3$ is the complement of $K_{c_{3}}+c_{2}K_{2}+c_{1}K_{1}\in\K$, $c_{3}=0$ or $c_{3}\ge3$, and $c_1+c_2+c_3=n_3$}
              \IF{$F_1,F_2,F_3$ do not satisfy Lemma~\ref{lem: min edge sizes differ by} on number of edges}
                \STATE \textbf{continue} with the next choice for $F_3$
              \ENDIF
              \IF{$F_1,F_2,F_3$ satisfy all six classes}
                \STATE add $\{F_1,F_2,F_3\}$ to the list of candidate triples
              \ENDIF
            \ENDFOR \COMMENT{ (choosing $F_3\in \K^c$)} \bigskip

          \ELSIF{$F_1,F_2\notin \S$}
            \STATE \COMMENT{We choose $F_3$ to be a forest of stars.  Note that this requires generating all partitions of $n_3$.}
            \FORALL{$F_3\in\S$ that is a forest of stars on $n_3$ vertices}
              \IF{$F_1,F_2,F_3$ do not satisfy Lemma~\ref{lem: min edge sizes differ by} on number of edges}
                \STATE \textbf{continue} with the next choice for $F_3$
              \ENDIF
              \IF{$F_1,F_2,F_3$ satisfy all six classes}
                \STATE add $\{F_1,F_2,F_3\}$ to the list of candidate triples
              \ENDIF
            \ENDFOR \COMMENT{ (choosing $F_3\in \S$)} \bigskip

          \ELSIF{$F_1,F_2\notin \S^c$}
            \STATE \COMMENT{We choose $F_3$ to be a complement of a forest of stars.  Note that this requires generating all partitions of $n_3$.}
            \FORALL{$F_3\in\S^c$ that is a complement of a forest of stars on $n_3$ vertices}
              \IF{$F_1,F_2,F_3$ do not satisfy Lemma~\ref{lem: min edge sizes differ by} on number of edges}
                \STATE \textbf{continue} with the next choice for $F_3$
              \ENDIF
              \IF{$F_1,F_2,F_3$ satisfy all six classes}
                \STATE add $\{F_1,F_2,F_3\}$ to the list of candidate triples
              \ENDIF
            \ENDFOR \COMMENT{ (choosing $F_3\in \S^c$)} \bigskip

          \ENDIF \COMMENT{ (which class $F_3$ should be chosen from)}
        \ENDFOR \COMMENT{ (choosing $n_3$)}
      \ENDFOR \COMMENT{ (choosing $F_2$)}
    \ENDFOR \COMMENT{ (choosing $F_1$)}
  \end{algorithmic}
\end{alg}

\subsubsection{Phase II: testing potential triples}

In the second phase of the algorithm, we test each candidate triple $\F$ generated by Phase~I to determine if it is indeed a DSF set.  Algorithm~\ref{alg:test_candidate_triples} describes the second phase in pseudocode.  First, we check that no graph in $\F$ is an induced subgraph of another graph in $\F$.  Second, we check whether $\F$ contains $\{2K_2,C_4\}$, as a minimal DSF triple cannot contain this DSF pair.  Third, we do a brute force search for an $\F$-breaking pair; this search is constrained to graphs with at most two more vertices by Lemma~\ref{lem: at most 2 more vtcs}.  If no $\F$-breaking pair is found, then $\F$ is a minimal DSF triple.

\medskip
\begin{alg}\label{alg:test_candidate_triples}
Algorithm for Phase II for testing each candidate triple.
\medskip
  \begin{algorithmic}[1]
    \FORALL{candidate triple $\F=\{F_1,F_2,F_3\}$}
      \IF{one of the graphs of $\F$ is induced in another}
        \STATE \COMMENT{$\F$ is not a minimal DSF triple.}
        \STATE continue with next choice of $\F$
      \ELSIF{$\F$ contains $\{2K_2,C_4\}$}
        \STATE \COMMENT{$\F$ is not a minimal DSF triple.}
        \STATE continue with next choice of $\F$
      \ENDIF \bigskip

      \FORALL{$F_i\in\F$}
        \FORALL{graph $H$ formed from $F_i$ by adding one or two vertices and some set of edges incident to the new vertices}
          \STATE \COMMENT{Note that $F_i$ is induced in $H$.}
          \FORALL{graph $H'$ formed from $H$ by performing a two switch involving the new vertices}
            \IF{$H'$ does not induce any member of $\F$}
              \STATE report $(H,H')$ as an $\F$-breaking pair
            \ENDIF
          \ENDFOR \COMMENT{ (choosing $H'$)}
        \ENDFOR \COMMENT{ (choosing $H$)}
      \ENDFOR \COMMENT{ (choosing $F_i$)}
      \IF{no breaking pair is found for $\F$}
        \STATE report $\F$ as a minimal DSF triple
      \ENDIF
    \ENDFOR \COMMENT{ (for each candidate triple)}
  \end{algorithmic}
\end{alg}

\subsection{Implementation}

The search algorithm was implemented in the mathematical software Sage~\cite{sage}. The run time was about $6.5$ minutes on a Fedora 14 Linux computer with an Intel Core i5 3.60GHz processor. The Sage code and output is available at the second author's webpage~\cite{webpage}.  Phase~I generated 3058 candidate triples, and the two triples $\{K_3+K_1, C_4, P_3+K_1\}$ and $\{K_{1,3}, 2K_2, \paw\}$ were the only DSF sets found by Phase~II.  About $95\%$ of the total time is spent in Phase~II searching for breaking pairs. Though some breaking pairs found in Phase~II occur for several candidate triples, $246$ distinct breaking pairs were used.  This suggests that there is no short way to rule out the non-DSF candidate triples that would eliminate the use of the computer program.

\subsection{Results}
\label{sec: tripleslist}


The computer search completely determines all minimal DSF triples containing graphs with at least four vertices.  To prove that the candidate triples with small graphs from Lemma~\ref{lem: triples with K3} and Corollary~\ref{cor: triples with K2K1} are in fact DSF sets, we use the concept of unigraphs, though we could also run Phase~II on these triples.  A \emph{unigraph} is a graph that is the unique realization, up to isomorphism, of its degree sequence. The authors and M. Kumbhat noted the following in~\cite{BarrusEtAl08}.

\begin{obs}[\cite{BarrusEtAl08}] \label{obs: unigraph-producing}
Given a set $\F$ of graphs, if every $\F$-free graph is a unigraph, then $\F$ is DSF.
\end{obs}

In~\cite{Barrus12} the first author established a criterion for recognizing such sets $\F$.  The graphs $R$, $\overline{R}$, $S$, and $\overline{S}$ are shown in Figure~\ref{fig: R and S}. The join of two graphs is indicated with the symbol $\vee$.
\begin{figure}
\centering
\includegraphics{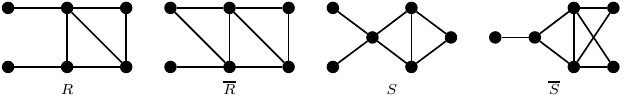}
\caption{The graphs $R$, $\overline{R}$, $S$, and $\overline{S}$ from Theorem~\ref{thm: unigraphic DSF}.}
\label{fig: R and S}
\end{figure}

\begin{thm}[\cite{Barrus12}]\label{thm: unigraphic DSF}
The $\F$-free graphs are all unigraphs if and only if every element of
\begin{multline*}
\{P_5, \mathrm{house}, K_2 +K_3, K_{2,3}, \textup{4-pan}, \textup{co-4-pan}, 2P_3, (K_2 +K_1) \vee (K_2 +K_1),\\
K_2 +P_4,2K_1 \vee P_4,K_2 +C_4,2K_1 \vee 2K_2,R,\overline{R}, S, \overline{S}\}
\end{multline*}
induces an element of $\F$.
\end{thm}

Combining our theoretical results and the output of the computer search, we obtain the full list of minimal DSF triples.

\begin{thm} \label{thm: tripleslist}
Following is the complete list of minimal DSF triples:
\begin{center}
\begin{tabular}{lcl}
$\{K_3, K_{1,3}, 2K_2\}$, & & $\{3K_1, K_1+K_3, C_4\}$,\\
$\{K_3, K_{1,3}, P_3+K_1\}$, & & $\{3K_1, K_1+K_3, \paw\}$,\\
$\{K_3, K_{2,3}, P_3+K_1\}$, & & $\{3K_1, K_2+K_3, \paw\}$,\\
$\{K_3, C_4, P_3+K_1\}$, & & $\{3K_1, 2K_2, \paw\}$,\\
$\{K_3+K_1, C_4, P_3+K_1\}$, & & $\{K_{1,3}, 2K_2, \paw\}$.
\end{tabular}
\end{center}
\end{thm}

\begin{proof}
From the output of Phases~I and II of the computer search, the only minimal DSF sets containing graphs with at least four vertices are $\{K_3+K_1, C_4, P_3+K_1\}$ and $\{K_{1,3}, 2K_2, \paw\}$.  Note that Phase~II provides a proof that these two triples are in fact DSF sets and ensures that they are minimal.

By Lemma~\ref{lem: triples with K3} and Corollary~\ref{cor: triples with K2K1}, the triples listed above are the only sets containing a graph with fewer than four vertices that can be DSF.  To show that each of these triples is DSF, it suffices to note that the condition in Theorem~\ref{thm: unigraphic DSF} holds.  Observe that none of these triples contain a DSF singleton or pair from Theorem~\ref{thm: pairs list} and hence all are minimal DSF triples.
\end{proof}

\section{Open questions} \label{sec: conclusion}

Using the results of Section~\ref{sec: minimal properties} on minimal DSF sets, we have generated the list of all minimal DSF triples in Theorem~\ref{thm: tripleslist}. A similar approach, again using an exhaustive search algorithm, could be used to find all minimal DSF sets with $k$ elements for any fixed $k$. However, because of practical limits on computational power, new ideas will be necessary to deal with $k$ much larger than $3$. These ideas might result from a deeper understanding of DSF sets in general, and minimal DSF sets in particular. We describe some questions here.

In Section~\ref{sec: minimal properties}, we showed that every DSF set has to contain a graph from each of several hereditary families. This leads to our first question.

\begin{ques}
What other hereditary families must be represented by elements in every DSF set?
\end{ques}

It appears possible that even a complete answer to this question would not suffice to characterize minimal DSF sets. The authors have shown, for instance, that not only must a DSF set contain a forest of stars, but this forest must appear among the elements having smallest maximum degree.

\begin{ques}
What relationships must hold between an element of a (minimal) DSF set and the set as a whole?
\end{ques}

One answer to this question was given in Theorem~\ref{thm: DSF iff D condition holds}, where DSF sets were characterized in terms of the sets $\mathcal{D}(G)$ corresponding to their elements. However, as shown in Example~\ref{exa: K_3-sieve}, determining a set $\mathcal{D}(G)$ is not a trivial matter, so necessary conditions such as the result on forests of stars above are in some cases preferable. Such results might be used productively in algorithms similar to Algorithm~\ref{alg:find_potential_triples} to rule out candidate DSF sets.

Examining Theorems~\ref{thm: pairs list} and~\ref{thm: tripleslist}, it is interesting to note that of the minimal DSF sets with three or fewer elements, all but $\{2K_2,C_4\}$ satisfy the hypothesis of Observation~\ref{obs: unigraph-producing}, that is, forbidding the subgraphs in these sets produces a class of graphs containing only unigraphs. Indeed, Observation~\ref{obs: unigraph-producing} has thus far been the most useful of the few known sufficient conditions for DSF sets. However, in light of Theorem~\ref{thm: unigraphic DSF}, ``unigraph-generating'' DSF sets are relatively quite few; no minimal DSF set containing a graph with at least seven vertices can be the set of forbidden subgraphs for a hereditary family of unigraphs. This leads to another question.

\begin{ques}
What properties are shared by (minimal) DSF sets $\mathcal{F}$ for which ``most'' $\mathcal{F}$-free graphs are not unigraphs?
\end{ques}

And more generally, given the number of necessary conditions we have found for DSF sets, we ask for results of the opposite type.

\begin{ques}
What easily tested conditions are sufficient to ensure that a set is DSF?
\end{ques}

Finally, we conclude with questions on finiteness. We showed in Section~\ref{sec: minimal properties} that each graph belongs to a minimal DSF set, and Theorem~\ref{thm: bound on min DSF size} ties the number of graphs in such a set to the size of graphs in the set. It follows that there must be arbitrarily large minimal DSF sets. However, all examples the authors are aware of at present are finite. Must it be so?

\begin{ques}
Does there exist an infinite minimal DSF set?
\end{ques}

In particular, the proof of Theorem~\ref{thm: every graph in minimal} showed that for any graph $G$, the set $\mathcal{D}(G)$ is a minimal DSF set. What if we restricted our attention to these sets?

\begin{ques}
For any graph $G$, must $\mathcal{D}(G)$ be a finite set?
\end{ques}

\section*{Acknowledgements}

The authors thank Mohit Kumbhat for fruitful discussions.

\end{document}